\documentclass[11pt]{amsart}
\usepackage{amscd,amssymb,amsmath,amsbsy,amsthm}
\usepackage{comment,color}
\usepackage[all]{xy}
\usepackage[colorlinks,plainpages,urlcolor=blue]{hyperref}
\usepackage[width=6.5in,height=8.0in,centering]{geometry}
\usepackage{fancyvrb,bbm}
\usepackage[mathscr]{euscript}
\usepackage{graphicx}
\usepackage{tikz}
\usepackage{xcolor}
\usepackage{booktabs}

\theoremstyle{definition}

\newtheorem{dfn}{Definition}
\newtheorem{rmk}{Remark}

\theoremstyle{plain}
\newtheorem{thm}{Theorem}[section]
\newtheorem{lem}[thm]{Lemma}
\newtheorem{cor}[thm]{Corollary}
\newtheorem{prop}[thm]{Proposition}

\DeclareMathOperator{\ShExt}{\mathscr{E}\kern -0.5pt xt}

\DeclareMathOperator{\ShHom}{\mathscr{H}\kern -0.5pt om}
\usepackage{amsaddr}
\begin{document}

\title[Dot Product Graph]%
{Properties of the Dot Product Graph \\of a Commutative Ring}

\author[Mohsen Mollahajiaghaei]{Mohsen Mollahajiaghaei}
\address{Department of Mathematics,
University of Western Ontario,\\ London, Ontario, Canada N6A 5B7}
\email{\href{mailto:mmollaha@uwo.ca}{mmollaha@uwo.ca}}
%\urladdr{\href{http://www.math.uiuc.edu/~schenck/}%
%{http://www.math.uiuc.edu/\~{}schenck}}

%\thanks{{$^1$}Partially supported by a grant from NSERC of Canada}
%\thanks{{$^2$}Partially supported by NSF DMS 1312071.}



\begin{abstract}
Let $R$ be a commutative ring with identity and $n\geq1$ be an integer. Let $R^{n}=R\times\cdots\times R~(n~times)$. The \textit{total dot product} graph, denoted by $TD(R,n)$ is a simple graph with elements of $R^{n}-\{(0,0,\ldots,0)\}$ as vertices, and two distinct vertices $\mathbf{x}$ and $\mathbf{y}$ are adjacent if and only if $\mathbf{x} \cdot \mathbf{y}=0\in R$, where $\mathbf{x} \cdot \mathbf{y}$ denotes the dot product of $\mathbf{x}$ and $\mathbf{y}$. In this paper, we find the structure of $TD(R\times S,n)$ with respect to the structure of $TD(R,n)$ and $TD(S,n)$. In addition, we find the degree of vertices of this graph. We determine when it is regular. Let $\mathbb{F}$ be a finite field. It is shown that if $TD(\mathbb{F},n)\simeq TD(R,m)$, then $n=m$ and $R\simeq\mathbb{F}$. A number of results concerning the domination number are also presented. Furthermore, we give some results on the clique and the independence number of $TD(R,n)$. It is shown that the ring $R$ is finite if and only if its independence number is finite. Finally, we classify all planar graphs within this class.
\end{abstract}
\maketitle
\setcounter{tocdepth}{1}
%---------------------------------------------------------------------------------------%
\section{Introduction}
%---------------------------------------------------------------------------------------%
There are many papers purporting to study the interplay between commutative rings
and combinatorics – typically, these involve starting with a ring and studying some
graph associated to it (e.g. zero-divisor graph, unitary Cayley graph). By virtue of their
definition, most of these graphs have a lot of symmetry, and hence lend themselves
well to the computation of various combinatorial invariants; this pursuit has attracted
the attention of many people in the last three decades, see \cite{Anderson,Anderson1,Badawi,Beck,domintotal,our paper,elect,Rocky}.

Let $R$ be a commutative ring with nonzero identity and $n \geq 1$ be an integer. 
Let $R^{n}=R\times \cdots \times R~(n~ times)$. Badawi \cite{Badawi} introduced the \textit{total dot product} graph, denoted by $TD(R,n)$, as a simple graph with elements of $R^{n}-\{(0,0,\ldots,0)\}$ as vertices, and two distinct vertices $\mathbf{x}$ and $\mathbf{y}$ are adjacent if and only if $\mathbf{x}\cdot\mathbf{y} =0 \in R$, where $\mathbf{x}\cdot \mathbf{y}$ denotes the dot product of $\mathbf{x}$ and $\mathbf{y}$. For example, figure (\ref{fig:sample1}) depicts $TD(\mathbb{Z}_2,3)$. In \cite{Badawi}, it was shown that the diameter of this graph for $n\geq 3$ is $3$. Also, for $n=2$, the diameter was determined. In addition, the girth of this graph was studied.

By the \textit{zero-divisor graph} $\Gamma(R)$ of $R$, we mean the graph with vertices $Z(R)-\{0\}$ such that there is an (undirected) edge between vertices $a$ and $b$ if and only if $a\neq b$ and $ab=0$.
For an arbitrary natural number $n$ and ring $R$, it can be easily seen that there exist $n$ mutually distinct copies of $\Gamma(R)$ in $TD(R,n)$.

Throughout this paper, we use $N(v)$ for the neighborhood of a vertex (i.e. the set of vertices adjacent to $v$). For a graph $G$, let $V(G)$ denote the set of vertices. The \textit{tensor product} of $G_{1}$ and $G_{2}$, $G_{1} \otimes G_{2}$, is the
graph with vertex set $V(G_{1}\otimes G_{2}):=V(G_{1})\times
V(G_{2})$, specified by putting $(u,v)$ adjacent to $(u',v')$ if
and only if $u$ is adjacent to $u'$ in $G_{1}$ and $v$ is
adjacent to $v'$ in $G_{2}$.

A set $D$ of vertices of a graph $G$ is said to be \textit{dominating} if every vertex of $V(G)-D$ is adjacent to a vertex of $D$, and the \textit{domination number} $\gamma(G)$ is the minimum number of vertices of a dominating set in $G$. For a given graph $G$ and a natural number $k$, the decision problem testing whether $\gamma(G)\leq k$ was shown to be NP-complete \cite{NPdomin}.

A subset $I$ of $V(G)$ is said to be \textit{independent} if any two vertices in that subset are pairwise non-adjacent. The \textit{independence number} of a graph $G$, denoted by $\alpha(G)$, is the maximum size of an independent set of vertices in $G$.

A \textit{clique} is a set of pairwise adjacent vertices in a graph. The \textit{clique number} of a graph $G$, denoted by $\omega(G)$, is the size of the largest clique of $G$.

A \textit{planar graph} is a graph that can be embedded in the plane, i.e., it can be drawn on the plane in such a way that its edges intersect only at their endpoints. In other words, it can be drawn in such a way that no edges cross each other. For basic terminology regarding graphs, we refer the reader to \cite{west}.

Throughout this paper, $R$ is a finite commutative ring with identity. Here $R^{\ast}$ and $U(R)$ stand for $R-\{0\}$ and invertible elements of $R$, respectively. A ring $R$ is said to be \textit{reduced} if $R$ has no nonzero nilpotent element. So, a finite commutative reduced ring $R$ is a finite product of finite fields. Let $\mathbf{a}=(a_1,\ldots,a_n)\in R^n$, by $||\mathbf{a}||$ we denote $a_{1}^2+a_{2}^2+\cdots+a_{n}^2$. We shall denote by $\mathbb{Z}_n$ the ring of integers modulo $n$. Let $e_i$ ($i=1,\ldots,n$) be the element in $R^n$ such that $j$-coordinate is 0 for $j \neq i$, and $i$-coordinate is 1.

%---------------------------------------------------------------------------------------%
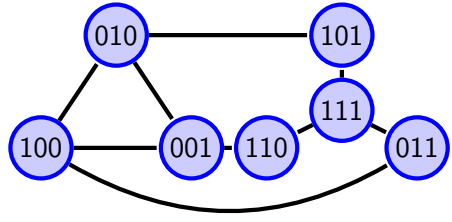
\begin{figure}
\begin{tikzpicture}[baseline=(current bounding box.center),scale=1, shorten >=0.8pt, auto, node distance=2cm, ultra thick]
   \begin{scope}[every node/.style={circle,draw=blue,fill=blue!20!,inner sep = 2pt, ,font=\sffamily}]  

\node (v1) at (3,3) {100};
    \node (v2) at (4,4.5) {010};
 \node (v3) at (5,3) {001};
 \node (v4) at (6,3) {110};
 \node (v5) at (7,4.5) {101};
         \node (v6) at (8,3) {011};    
             \node (v7) at (7,3.5) {111};    
    \draw  (v1) edge  (v2);
    \draw  (v1) edge  (v3);
    \draw  (v2) edge  (v3);

\draw  (v1) edge [bend right] (v6);
\draw  (v2) edge  (v5);
\draw  (v3) edge  (v4);

\draw  (v7) edge  (v4);
\draw  (v7) edge  (v5);
\draw  (v7) edge  (v6);

   \end{scope}
\end{tikzpicture}
\caption{$TD(\mathbb{Z}_2,3)$}\label{fig:sample1}
\end{figure}
%---------------------------------------------------------------------------------------%
%---------------------------------------------------------------------------------------%
%---------------------------------------------------------------------------------------%
Let $\overline{TD}(R,n)$ be the graph whose vertex set is $R^n$, and in which $\mathbf{x}$ is adjacent to $\mathbf{y}$ if and only if $\mathbf{x}\cdot \mathbf{y}=0$. Therefore we have loops. Let $G=\overline{TD}(R,n)$.  Remove the vertex with maximum degree and vertices with loops, so the new graph is $\overline{TD}(R,n)$. Thus, $\overline{TD}(R,n)$ and $TD(R,n)$ have a lot of similarities. Then, it is worthwhile to study $\overline{TD}(R,n)$. Figure (\ref{fig:sample2}) shows $\overline{TD}(\mathbb{Z}_2,3)$.

%---------------------------------------------------------------------------------------%
%---------------------------------------------------------------------------------------%
In section 2, we proceed with the study of the graph $\overline{TD}(R,n)$. In addition, we study the degree of vertices in $TD(R,n)$. Finally, we prove that if $TD(\mathbb{F},n)\simeq TD(R,m)$, where $\mathbb{F}$ is a finite field, then $n=m$ and $R\simeq\mathbb{F}$. Section 3 is devoted to the study of the domination number of $TD(R,n)$. We find the domination number of $TD(\mathbb{F},n)$. We give some upper bounds for an arbitrary ring. Moreover, we will discuss the domination number of $TD(R,n)$ for infinite rings. In the fourth section, we will look at the clique and independence number. The last section in this paper lists all planar graphs within this class.
%---------------------------------------------------------------------------------------%
%---------------------------------------------------------------------------------------%
\section{Degree sequence and $\overline{TD}(R,n)$}
%---------------------------------------------------------------------------------------%
It is natural to relate $\overline{TD}(R\times S,n)$ to $\overline{TD}(R,n)$ and $\overline{TD}(S,n)$. The first theorem provides the relation between these graphs.
%---------------------------------------------------------------------------------------%
%---------------------------------------------------------------------------------------%
%---------------------------------------------------------------------------------------%
\begin{thm}\label{tensorprod}
Let $R$ and $S$ be arbitrary rings. Then $\overline{TD}(R\times S,n)\simeq \overline{TD}(R,n) \otimes \overline{TD}(S,n)$.
\end{thm}
%---------------------------------------------------------------------------------------%
\begin{proof}
Let $G=\overline{TD}(R\times S,n)$. The vertex $\mathbf{a}=((r_1,s_1),(r_2,s_2),\ldots,(r_n,s_n))$ in $G$ is adjacent to $\mathbf{b}=((r'_1,s'_1),(r'_2,s'_2),\ldots,(r'_n,s'_n))$ if and only if $\mathbf{a}.\mathbf{b}=(\sum_{i=1}^{n}r_i r'_{i},\sum_{i=1}^{n}s_i s'_{i})=(0,0)$. Equivalently, $\mathbf{r}=(r_1,r_2,\ldots,r_n)$ is adjacent to $\mathbf{r'}=(r'_1,r'_2,\ldots,r'_n)$ in $\overline{TD}(R,n)$ and  $\mathbf{s}=(s_1,s_2,\ldots,s_n)$ is adjacent to $\mathbf{s'}=(s'_1,s'_2,\ldots,s'_n)$ in $\overline{TD}(S,n)$, which proves the theorem.
\end{proof}
%---------------------------------------------------------------------------------------%
%---------------------------------------------------------------------------------------%
\begin{figure}
\begin{center}
\begin{tikzpicture}[baseline=(current bounding box.center),scale=1, shorten >=0.9pt, auto, node distance=2cm, ultra thick]
   \begin{scope}[every node/.style={circle,draw=blue,fill=blue!20!,inner sep = 2 pt, ,font=\sffamily}]  

\node (v1) at (3,3) {100};
    \node (v2) at (4,5.5) {010};
 \node (v3) at (5,3) {001};
 \node (v4) at (6,3) {110};
 \node (v5) at (7,5.5) {101};
         \node (v6) at (8,3) {011};    
             \node (v7) at (7,3.5) {111};    
              \node (v8) at (5.3,4.3) {000}; 
    \draw  (v1) edge  (v2);
    \draw  (v1) edge  (v3);
    \draw  (v2) edge  (v3);

\draw  (v1) edge [bend right] (v6);
\draw  (v2) edge  (v5);
\draw  (v3) edge  (v4);

\draw  (v7) edge  (v4);
\draw  (v7) edge  (v5);
\draw  (v7) edge  (v6);

\draw  (v8) edge  (v1);
\draw  (v8) edge  (v2);
\draw  (v8) edge  (v3);
\draw  (v8) edge  (v4);
\draw  (v8) edge  (v5);
\draw  (v8) edge  [bend left] (v6);
\draw  (v8) edge  (v7);

\draw (v4) to [in=70, out=90, loop] ();
\draw (v5) to [in=60, out=100, loop] ();
\draw (v6) to [in=60, out=100, loop] ();
\draw (v8) to [in=60, out=100, loop] ();

   \end{scope}
\end{tikzpicture}
\end{center}
\caption{$\overline{TD}(\mathbb{Z}_2,3)$.}\label{fig:sample2}
\end{figure}
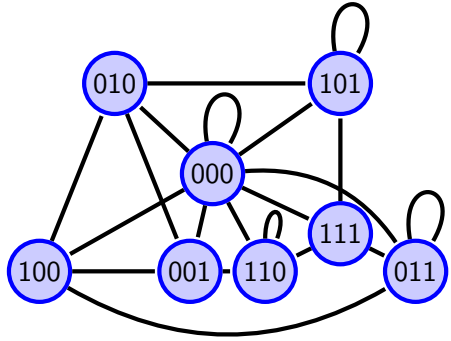
%---------------------------------------------------------------------------------------%
If $R$ is a finite commutative ring, then $R \simeq R_{1}\times \cdots\times
R_{t}$ where each $R_{i}$ is a finite commutative local ring with
maximal ideal $M_{i}$, by Theorem 8.7 of \cite {Atiyah}. Hence, by the aforementioned theorem, $\overline{TD}(R,n)\simeq \bigotimes_{i=1}^{t}\overline{TD}(R_{i},n)$.
%---------------------------------------------------------------------------------------%
\begin{rmk}
Let $R$ and $S$ be arbitrary rings. Theorem \ref{tensorprod} immediately tells us that the number of loops in $\overline{TD}(R\times S,n)$ is product of the number loops of $\overline{TD}(R,n)$ and $\overline{TD}(S,n)$.
\end{rmk}
%---------------------------------------------------------------------------------------%
%---------------------------------------------------------------------------------------%
\begin{rmk}
Let $O(R,n)$ be the number of non-trivial solutions of the equation $x_1^2+x_2^2+\cdots+x_n^2=0$. Then the number of loops in $\overline{TD}(R,n)$ equals to $O(R,n)+1$. By Exercise 19 in Chapter 8 of \cite{ireland}, we know that if $n$ is an odd number and $\mathbb{F}$ is the field of prime order $p$, then the number of loops in $\overline{TD}(\mathbb{F},n)$ is $p^{n-1}$.

Let $\mathbb{F}$ be a finite field of characteristic 2.  Since $x_{1}^{2}+x_{2}^{2}+\cdots+x_{n}^2=(x_1+\cdots+x_{n})^2$, it follows that the number of loops of $\overline{TD}(\mathbb{F},n)$ is $|\mathbb{F}|^{n-1}$.
\end{rmk}
%---------------------------------------------------------------------------------------%
%---------------------------------------------------------------------------------------%
\begin{thm}[Chevalley--Warning]
Let $\mathbb{F}$ be the field with $q=p^{\alpha}$ elements, where $p$ is a prime number. If $f(x_1,\ldots,x_n)\in \mathbb{F}[x_1,\ldots,x_n]$ and $\deg(f)<n$, then $| \{ (a_1,\ldots,a_n)\in \mathbb{F}^{n} \mid f(a_1,\ldots,a_n)=0  \}| \equiv 0~(\bmod~p)$.
\end{thm}
%---------------------------------------------------------------------------------------%
Let $n>2$. As a consequence of the Chevalley-Warning theorem, there exists a non-trivial loop in $\overline{TD}(\mathbb{F},n)$. Moreover, the number of loops is divisible by the characteristic of $\mathbb{F}$.

%---------------------------------------------------------------------------------------%
It is readily known that the equation $x^{2}+y^2=0$ has a non-trivial solution in $\mathbb{F}$ if and only if $|\mathbb{F}|\not\equiv 3~(\bmod~4)$. Therefore, we have:
\begin{equation*}O(\mathbb{F},2)=\left\{
\begin{array}{cc}
0 & |\mathbb{F}|\equiv 3~(\bmod~4) \\
2(|\mathbb{F}|-1) & |\mathbb{F}|\equiv 1~(\bmod~ 4)\\
|\mathbb{F}|-1 & |\mathbb{F}|~\textrm{is even}.\
\end{array} \right.
\end{equation*}
%---------------------------------------------------------------------------------------%
%---------------------------------------------------------------------------------------%
In the remainder of this section, we will restrict our attention to the degree of vertices, and isomorphism problem for $TD(R,n)$.
%---------------------------------------------------------------------------------------%
Let $\mathbf{a} \in R^n$ and $Z( \mathbf{a})=\{ \mathbf{b}; \mathbf{a}\cdot \mathbf{b}=0\}$. Hence, $N(\mathbf{a})=Z(\mathbf{a}) - \{0,\mathbf{a} \}$. Obviously, $Z(\mathbf{a})$ is a $R$-submodule of $R^n$, and $\deg \mathbf{a}=|N(\mathbf{a})|-1$ if $||\mathbf{a}||\neq 0$ and $\deg \mathbf{a}=|N(\mathbf{a})|-2$ otherwise.\\
%---------------------------------------------------------------------------------------%
%---------------------------------------------------------------------------------------%
%---------------------------------------------------------------------------------------%
%---------------------------------------------------------------------------------------%
%---------------------------------------------------------------------------------------%
In the following theorem, we will find the degree of the vertex $\mathbf{a}=(a_{1},\ldots,a_{n})\in R^n$, if at least one coordinate is invertible.
%---------------------------------------------------------------------------------------%
%---------------------------------------------------------------------------------------%
\begin{thm}
Let $R$ be a finite ring with nonzero identity. Let $\mathbf{a}=(a_{1},\ldots,a_{n})\in R^n$ such that there exists $1\leq i \leq n$ in such a way that $a_i$ is invertible. Then for the degree of $\mathbf{a}$ in the graph $TD(R,n)$, we have the following:
\begin{itemize}
\item[(a)]
If $|| \mathbf{a} || \neq 0$, then $\deg \mathbf{a}=|R|^{n-1}-1$.
\item[(b)]
If $|| \mathbf{a} ||=0$, then $\deg \mathbf{a}=|R|^{n-1}-2$.
\end{itemize}
\end{thm}
%---------------------------------------------------------------------------------------%
%---------------------------------------------------------------------------------------%
\begin{proof}
There is no loss of generality in assuming that $a_{n}$ is invertible. Thus, for any choice of $x_{1},x_2,\ldots,x_{n-1}$, there exists a unique $x_{n}\in R$ such that $\sum_{j=1}^{n-1} a_{j}x_{j}=-a_{n}x_n$. Then there exists exactly $|R|^{n-1}$ elements in $R^{n}$ in such a way that $\mathbf{a}\cdot \mathbf{x}=0$, which completes the proof.
\end{proof}
%---------------------------------------------------------------------------------------%
%---------------------------------------------------------------------------------------%
%---------------------------------------------------------------------------------------%
%---------------------------------------------------------------------------------------%
%---------------------------------------------------------------------------------------%
Consequently, in the case of finite field, we conclude the following corollary.
%---------------------------------------------------------------------------------------%
\begin{cor}\label{field}
Let $\mathbb{F}$ be a finite field with $q$ elements. Then the following hold:
\begin{itemize}
\item[(a)]
If $n>2$, then $TD(\mathbb{F},n)$ is a semi-regular graph with degrees $q^{n-1}-1$ and $q^{n-1}-2$.
\item[(b)]
If $n=2$ and $q\equiv 3~(\bmod~4)$, then $TD(\mathbb{F},n)$ is a regular graph of valency $q-1$.
\item[(c)]
If $n=2$ and $q\equiv 1~(\bmod~4)$, then $TD(\mathbb{F},n)$ is a semi-regular graph with degrees $q-1$ and $q-2$.
\item[(d)]
If $n=2$ and $char(\mathbb{F})=2$, then $TD(\mathbb{F},n)$ is a semi-regular graph with degrees $q-1$ and $q-2$.
\item[(e)]
If $n=1$, then $TD(\mathbb{F},n)$ is the empty graph with $q-1$ vertices.
\end{itemize}
\end{cor}
%---------------------------------------------------------------------------------------%
%---------------------------------------------------------------------------------------%
%---------------------------------------------------------------------------------------%
In \cite{Badawi} it was shown that if $R$ is an integral domain, then $TD(R,2)$ is disconnected. In the next theorem we find the structure of $TD(\mathbb{F},2)$.
%---------------------------------------------------------------------------------------%
\begin{thm}\label{n2}
Let $\mathbb{F}$ be a finite field. Then $TD(\mathbb{F},2)$ is disconnected, and
\begin{itemize}
\item[(a)]
If $O(\mathbb{F},2)=0$, then the number of connected component is $\dfrac{|\mathbb{F}|+1}{2}$. Moreover, $TD(\mathbb{F},2)$ is disjoint union of  $\dfrac{|\mathbb{F}|+1}{2}$ complete bipartite graphs $K_{|\mathbb{F}|-1,|\mathbb{F}|-1}$.
\item[(b)]
The graph $TD(\mathbb{F},2)$ is disjoint union of $\dfrac{O(2,\mathbb{F})}{|\mathbb{F}|-1}$ complete graphs of size $|\mathbb{F}|-1$ and \newline $\dfrac{|\mathbb{F}|^2-1-O(2,\mathbb{F})}{2(|\mathbb{F}|-1)}$ complete bipartite graphs $K_{|\mathbb{F}|-1,|\mathbb{F}|-1}$.
\end{itemize}
\end{thm}
%---------------------------------------------------------------------------------------%
%---------------------------------------------------------------------------------------%
\begin{proof}
Let $(a,b)$ be a vertex in $TD(\mathbb{F},2)$. We have two cases:
\begin{itemize}
\item[(1)]
If $a^{2}+b^2\neq 0$. Let $A_{1}=\{(ra,rb)\mid r\in \mathbb{F}^{\ast} \}$ and $A_{2}=\{(-rb,ra)\mid r\in \mathbb{F}^{\ast}\}$. Obviously, the graph induced by $A_1 \cup A_2$ is isomorphic to the complete bipartite graph $K_{|\mathbb{F}|-1,|\mathbb{F}|-1}$. Then it is a connected component by Corollary \ref{field}.
\item[(2)]
If $a^{2}+b^2=0$. Then the graph induced by $W=\{(ra,rb)\mid r\in \mathbb{F}^{\ast}\}$ is a clique of size $|\mathbb{F}|-1$. Thus, it is a connected component by Corollary \ref{field}.
\end{itemize}
\end{proof}
%---------------------------------------------------------------------------------------%
%---------------------------------------------------------------------------------------%
%---------------------------------------------------------------------------------------%
The next theorem shows that if $\mathbb{F}$ and $\mathbb{E}$ are different fields or $m\neq n$, then the graph $TD(\mathbb{F},n)$ is not isomorphic to the graph $TD(\mathbb{E},m)$.
%---------------------------------------------------------------------------------------%
%---------------------------------------------------------------------------------------%
%---------------------------------------------------------------------------------------%
%---------------------------------------------------------------------------------------%
\begin{thm}\label{meydan}
Let $\mathbb{F}$ and $\mathbb{E}$ be finite fields, and let $m,n$ be integers. If $TD(\mathbb{F},n)\simeq TD(\mathbb{E},m)$, then $m=n$ and $\mathbb{F}\simeq\mathbb{E}$.
\end{thm}
%---------------------------------------------------------------------------------------%
\begin{proof}
Let $|\mathbb{F}|=q$ and $|\mathbb{E}|=r$. The number of vertices of $TD(\mathbb{F},n)$ and $TD(\mathbb{E},m)$ are $q^{n}-1$ and $r^{m}-1$, respectively. Therefore,
\begin{equation}\label{1}
q^{n}-1=r^{m}-1.
\end{equation}
The graphs $TD(\mathbb{F},n)$ and $TD(\mathbb{E},m)$ are regular or semi-regular graphs. The maximum degree of $TD(\mathbb{F},n)$ and $TD(\mathbb{E},m)$ are $q^{n-1}-1$ and $r^{m-1}-1$, respectively. Hence,
\begin{equation}\label{2}
q^{n-1}-1=r^{m-1}-1.
\end{equation}
Combining equations (\ref{1}) and (\ref{2}), we can see that $n=m$ and $q=r$.
\end{proof}
%---------------------------------------------------------------------------------------%
%---------------------------------------------------------------------------------------%
The next theorem deals with the degree of vertices for reduced rings.
%---------------------------------------------------------------------------------------%
\begin{thm}
Let $R=\mathbb{F}_{1}\times \cdots \times \mathbb{F}_{t}$, where $\mathbb{F}_{i}$ is a field for each $i=1,\ldots,t$. Then the degree of $\mathbf{a}=\big(  (a_{11},\ldots,a_{1t}),\ldots,(a_{n1},\ldots,a_{nt})  \big)$ is
\begin{equation*}\left\{
\begin{array}{cc}
\dfrac{|R|^{n}}{\prod_{i=1}^{t}|\mathbb{F}_{i}|^{\tau_i}}-1& \textrm{\rm{if}}~||\mathbf{a}||\neq 0\\
\\
\dfrac{|R|^{n}}{\prod_{i=1}^{t}|\mathbb{F}_{i}|^{\tau_i}}-2& \textrm{\rm{if}}~||\mathbf{a}||= 0,\\
\end{array} \right.
\end{equation*}
where,
\begin{equation*}\tau_i=\left\{
\begin{array}{cc}
0 & \textrm{if}~(a_{1i}, a_{2i},\ldots ,a_{ni})= \mathbf{0}\\
1 & \rm{otherwise}.\
\end{array} \right.
\end{equation*}
In particular, the minimum degree of $TD(R,n)$ is either $|R|^{n-1}-1$ or $|R|^{n-1}-2$.
\end{thm}
%---------------------------------------------------------------------------------------%
%---------------------------------------------------------------------------------------%
\begin{proof}
Let $\mathbf{a}=\big(  (a_{11},\ldots,a_{1t}),\ldots,(a_{n1},\ldots,a_{nt})  \big) \in R^{n}$. Then $\mathbf{b}=\big(  (b_{11},\ldots,b_{1t}),\ldots,(b_{n1},\ldots,b_{nt})  \big)$ is adjacent to $\mathbf{a}$ if the following system of equations is satisfied:
\begin{equation*}\left\{
\begin{array}{c}
a_{11}b_{11}+a_{21}b_{21}+\cdots +a_{n1}b_{n1}=0 \\
a_{12}b_{12}+a_{22}b_{22}+\cdots +a_{n2}b_{n2}=0 \\
\vdots\\
a_{1t}b_{1t}+a_{2t}b_{2t}+\cdots +a_{nt}b_{nt}=0.  \
\end{array} \right.
\end{equation*}
Equations are independent, so the number of solutions is $\dfrac{\prod_{i=1}^{t}|\mathbb{F}_{i}|^{n}}{\prod_{i=1}^{t}|\mathbb{F}_{i}|^{\tau_i}}$, which completes the proof.
\end{proof}
%---------------------------------------------------------------------------------------%
%---------------------------------------------------------------------------------------%
\begin{rmk}
Let $\mathbf{g}= \big( (a_1,b_1),(a_2,b_2),\ldots,(a_n,b_n)\big) \in (R\times S)^{n}$. Let $\mathbf{a}=(a_1,\ldots,a_n)\in R^n$ and $\mathbf{b}=(b_1,\ldots,b_n)\in S^n$. Let $\mathbf{a}\neq \mathbf{0}$ and $\mathbf{b}\neq \mathbf{0}$. Then the degree of $\mathbf{g}$ in $TD(R\times S,n)$ is:
\begin{equation*}\left\{
\begin{array}{cc}
(1+\deg_{R} \mathbf{a})(1+\deg_{S} \mathbf{b})-1& \textrm{if}~  ||\mathbf{a}||\neq 0~\textrm{and} ~||\mathbf{b}||\neq 0\\
(2+\deg_{R} \mathbf{a})(1+\deg_{S} \mathbf{b})-1& \textrm{if}~ ||\mathbf{a}||= 0~\textrm{and} ~||\mathbf{b}||\neq 0\\
(1+\deg_{R} \mathbf{a})(2+\deg_{S} \mathbf{b})-1& \textrm{if}~  ||\mathbf{a}||\neq 0~\textrm{and} ~||\mathbf{b}||= 0\\
(2+\deg_{R} \mathbf{a})(2+\deg_{S} \mathbf{b})-2& \textrm{if}~  ||\mathbf{a}|| = 0~\textrm{and} ~||\mathbf{b}||= 0,\\
\end{array} \right.
\end{equation*}
where $\deg_{R} \mathbf{a}$ and $\deg_{S} \mathbf{b}$ denote the degree of $\mathbf{a}$ and $\mathbf{b}$ in $TD(R,n)$ and $TD(S,n)$, respectively.\\
If $\mathbf{a}= \mathbf{0}$ and $\mathbf{b}\neq \mathbf{0}$, then the degree of $\mathbf{g}$ in $TD(R\times S,n)$ is:
\begin{equation*}\left\{
\begin{array}{cc}
|R|^{n}(1+\deg_{S}\mathbf{b})-1& \textrm{if}~||\mathbf{b}||\neq 0\\
|R|^n(2+\deg_{S} \mathbf{b})-2& \textrm{if}~||\mathbf{b}||= 0.\\
\end{array} \right.
\end{equation*}
If $\mathbf{a}\neq  \mathbf{0}$ and $\mathbf{b}= \mathbf{0}$, then the degree of $\mathbf{g}$ in $TD(R\times S,n)$ is:
\begin{equation*}\left\{
\begin{array}{cc}
|S|^{n}(1+\deg_{R} \mathbf{a})-1& \textrm{if}~||\mathbf{a}||\neq 0\\
|S|^n(2+\deg_{R} \mathbf{a})-2& \textrm{if}~||\mathbf{a}||= 0.\\
\end{array} \right.
\end{equation*}
\end{rmk}
%---------------------------------------------------------------------------------------%
%---------------------------------------------------------------------------------------%
\begin{thm}
Let $R=\mathbb{F}_{1}\times \cdots \times \mathbb{F}_{t}$, $S=\mathbb{E}_{1}\times \cdots \times \mathbb{E}_{s}$, where $\mathbb{F}_{i}$ and $\mathbb{E}_{j}$ are fields for each $i=1,\ldots,t$ and $j=1,\ldots,s$. Let $m,n$ be integers. If $TD(R,n)\simeq TD(S,m)$. Then $m=n$ and $|R|=|S|$.
\end{thm}
%---------------------------------------------------------------------------------------%
%---------------------------------------------------------------------------------------%
\begin{proof}
Since $TD(R,n)\simeq TD(S,m)$, we have $$|R|^{n}=|S|^{m}.$$
The minimum degree of $TD(R,n)$ is either $|R|^{n-1}-1$ or $|R|^{n-1}-2$. Also, the minimum degree of $TD(S,m)$ is either $|S|^{m-1}-1$ or $|S|^{m-1}-2$. Therefore, we can reduce to two cases:
\begin{itemize}
\item[(i)] If $$|R|^{n-1}-1=|S|^{m-1}-2,$$
then we get
$$|S|^{m-1}(|R|-|S|)=|R|.$$
Thus,
$$|S|^{n(m-1)}(|R|-|S|)^{n}=|S|^{m}.$$
Hence,
$$|S|^{(n-1)(m-1)-1}(|R|-|S|)^{n}=1.$$
It means that $m=2,n=2$ or either $m$ or $n$ is 1. If $m=n=2$, then $|R|^{2}=|S|^{2}$ and $|R|=|S|+1$, which cannot be hold. If $n=1$, then the graph $TD(R,n)$ has an isolated vertex but $TD(S,m)$ has no isolated vertex.
\item[(ii)] If $$|R|^{n-1}-1=|S|^{m-1}-1.$$
Similar to the proof of Theorem \ref{meydan}, we can get $m=n$ and $|R|=|S|$.
\end{itemize}
\end{proof}
%---------------------------------------------------------------------------------------%
The next theorem shows that for a field $\mathbb{F}$, the graph $TD(\mathbb{F},n)$, can be determined uniquely among all rings.
%---------------------------------------------------------------------------------------%
\begin{thm}\label{fieldDS}
Let $\mathbb{F}$ be a finite field and $R$ be a ring. Let $m,n$ be integers. If $TD(\mathbb{F},n)\simeq TD(R,m)$. Then $m=n$ and $R\simeq \mathbb{F}$.
\end{thm}
%---------------------------------------------------------------------------------------%
%---------------------------------------------------------------------------------------%
\begin{proof}
First we prove that $R$ must be a field. On the contrary, assume that $R$ is not a field. Let $d_1$ and $d_2$ be two vertex degree of the graph $TD(\mathbb{F},n)$. Then by Corollary \ref{field}, we have 
\begin{equation}\label{field1}
|d_1-d_2|\in \{ 0,1 \}.
\end{equation}
Let $R$ be a ring which is not a field. Hence there exists a non-zero zero divisor in $R$, say $a$. Let $b$ be a non-zero element of $R$ such that $ab=0$. Obviously, degree of $\mathbf{1}=(1,1,\ldots,1)$ is $|R|^{m-1}-2$ if $m$ is divisible by $char(R)$, and $|R|^{m-1}-1$, otherwise. Let $\mathbf{a}=(a,a,\ldots,a)$. Thus, $\mathbf{a}$ is adjacent to $(b_1,\ldots,b_{n})$ whenever either $b_{1}+\cdots+b_{n}=0$ or $b_{1}+\cdots+b_{n}=b$. Then $\deg ( \mathbf{a} )\geq 2 |R|^{m-1}-2$. Obviously, $2 |R|^{m-1}-2>|R|^{m-1}-2$, which contradicts Formula (\ref{field1}). The rest of the proof is clear by Theorem \ref{meydan}.
\end{proof}
%---------------------------------------------------------------------------------------%
\begin{rmk}
By Corollary \ref{field} and Theorem \ref{fieldDS}, we can classify all rings $R$ and integers $n$, so that the graph $TD(R,n)$ is regular.
\end{rmk}
%---------------------------------------------------------------------------------------%
%---------------------------------------------------------------------------------------%
\section{Domination number}
%---------------------------------------------------------------------------------------%
Let $G$ be a graph. If $G$ has no isolated vertices, then $\gamma(G)\leq \dfrac{n}{2}$. It is easy to see that for $k$-regular graph, $\gamma(G)\geq \dfrac{n}{k+1}$. The domination number of a graph and its many variations have been extensively studied in the literature \cite{haynes1}.\\
%---------------------------------------------------------------------------------------%
The next result, which is due to Meki{\v{s}} \cite{Mekis}, gives
a lower bound for domination number of tensor products of graphs.
%---------------------------------------------------------------------------------------%
\begin{thm}\cite{Mekis}
Let $G$ and $H$ be simple graphs. Then
\begin{equation*}\label{zarbedomin}
\gamma(G\otimes H) \geq \gamma(G)+\gamma(H)-1.
\end{equation*}
\end{thm}
%---------------------------------------------------------------------------------------%
%---------------------------------------------------------------------------------------%
It is of interest to find the domination number of graphs associated to rings, see \cite{domintotal} and \cite{dominzero}.
In this section, we would like to study the domination number of $\gamma(TD(R,n))$. It is easy to see that 
\begin{equation}\label{dominineqn}
\gamma(TD(R,n))\leq \gamma(TD(R,n-1)).
\end{equation}
%---------------------------------------------------------------------------------------%
%---------------------------------------------------------------------------------------%
In the next theorem we find the domination number of $TD(\mathbb{F},n)$.
%---------------------------------------------------------------------------------------%
%---------------------------------------------------------------------------------------%
\begin{thm}
Let $\mathbb{F}$ be a field with $q$ elements. Let $n>1$ be an integer. Then
\begin{equation}\gamma(TD(\mathbb{F},n))=\left\{
\begin{array}{cc}
2 &~\textrm{\rm{if}}~ \mathbb{F}\simeq\mathbb{Z}_2~\textrm{and}~ n=3,\\
q+1 & \textrm{\rm{otherwise}}.\
\end{array} \right.
\end{equation}
\end{thm}
%---------------------------------------------------------------------------------------%
%---------------------------------------------------------------------------------------%
\begin{proof}
If $\mathbb{F}\simeq\mathbb{Z}_2$ and $n=3$, then by Figure (\ref{fig:sample1}) one can easily check that $\gamma(TD(\mathbb{F},n))=2$.
By Theorem \ref{n2}, we can see that $\gamma(TD(\mathbb{F},2))=q+1$. Let $D=\{(a,1,0,\ldots,0)\mid a\in \mathbb{F} \} \cup \{ (1,0,0,\ldots,0) \}$. It is fairly easy to see that $D$ is a dominating set. Let $\mathbb{F}\neq\mathbb{Z}_2$ or $n\neq 3$, we prove that $\gamma(TD(\mathbb{F},n))$ cannot be less than $q+1$. On the contrary, assume that $D=\{\mathbf{d_{1}},\ldots,\mathbf{d_{q}} \}$ is a dominating set for $TD(\mathbb{F},n)$. By Corollary \ref{field}, each $\mathbf{d_{i}}$ can dominate at most $q^{n-1}$ vertices. Obviously, the system of equations
\begin{equation*}\left\{
\begin{array}{c}
\mathbf{d_{1}}\cdot\mathbf{x}=0\\
\mathbf{d_{2}}\cdot\mathbf{x}=0\
\end{array} \right.
\end{equation*}
%$$\mathbf{d_{1}}\cdot\mathbf{x}=0$$
%$$\mathbf{d_{2}}\cdot\mathbf{x}=0,$$
has more than $q^{n-2}-1>1$ non-trivial solutions. Therefore, the set $D$ dominates at most $q(q^{n-1})-2$ vertices, which means that $D$ is not a dominating set for $TD(R,n)$.
\end{proof}
%---------------------------------------------------------------------------------------%
The aforementioned theorem shows that inequality (\ref{dominineqn}) can be strict or can turn into equality.
%---------------------------------------------------------------------------------------%
%---------------------------------------------------------------------------------------%
\begin{rmk}
Let $R$ be a ring which is not a field. Let $r$ be a non-zero non-invertible element of $R$. Thus the equation $rx+1=0$ has no solution in $R$. Then $D=\{(a,1,0,\ldots,0)\mid a\in R \} \cup \{ (1,0,0,\ldots,0) \}$ is not a dominating set. 
\end{rmk}
%---------------------------------------------------------------------------------------%
%---------------------------------------------------------------------------------------%
The next two theorems give upper bounds for the domination number.
%---------------------------------------------------------------------------------------%
\begin{thm}
Let $R$ be a finite ring which is not a field. Let $n>1$ be an integer. Then $\gamma(TD(R,n))\leq |R-U(R)|^2-1$.
\end{thm}
%---------------------------------------------------------------------------------------%
%---------------------------------------------------------------------------------------%
\begin{proof}
We show that $\{(r,s,0,\ldots,0)\mid r,s \in R-U(R) ~and~ (r,s)\neq (0,0)\}$ is a dominating set for $TD(R,n)$. Let $\mathbf{a}=(a_1,a_2,\ldots,a_n)\in R^{n}$. We have three cases:
\begin{itemize}
\item[(i)]
If $a_1,a_2$ are invertible. Let $z\in Z(R)-\{0\}$. Then $\mathbf{a}$ is adjacent to $(-a_2z,a_1z,0,\ldots,0)$.
\item[(ii)]
If $a_1$ is invertible but $a_2$ is not invertible. Hence, there exists $z\in Z(R)-\{0\}$ such that $za_2=0$. Then $\mathbf{a}$ is adjacent to $(0,z,0,\ldots,0)$.
\item[(iii)]
If $a_1,a_2$ both are not invertiable. In this case, $\mathbf{a}$ is adjacent to $(-a_2,a_1,\ldots,0)$.
\end{itemize}
\end{proof}
%---------------------------------------------------------------------------------------%
%---------------------------------------------------------------------------------------%
\begin{thm}
Let $R$ be a finite ring which is not a field. Let $n>1$ be an integer. Then $\gamma(TD(R,n))\leq |R-U(R)|+|R|-2$.
\end{thm}
%---------------------------------------------------------------------------------------%
\begin{proof}
Let $A_{1}=\{(r,0,0,\ldots,0)\mid r\in R-U(R) ~and~ r\neq 0\}$ and $A_{2}=\{(0,s,0,\ldots,0)\mid s\in R-U(R) ~and~ r\neq 0\}$ and $A_{3}= \{(u,1,0,\ldots,0)\mid u\in U(R)\}$. We show that $A_{1}\cup A_{2}\cup A_{3}$ is a dominating set for $TD(R,n)$. Let $\mathbf{a}=(a_1,a_2,\ldots,a_n)\in R^{n}$. We have three cases:
\begin{itemize}
\item[(i)]
If $a_1,a_2$ are invertiable. Then $\mathbf{a}$ is adjacent to $(-a_2a_{1}^{-1},1,0,\ldots,0)$.
\item[(ii)]
If $a_1$ is not invertible. Hence, there exists $z\in Z(R)-\{0\}$ such that $za_1=0$. Then $\mathbf{a}$ is adjacent to $(z,0,0,\ldots,0)$.
\item[(iii)]
If $a_2$ is not invertible. Hence, there exists $z\in Z(R)-\{0\}$ such that $za_2=0$. Then $\mathbf{a}$ is adjacent to $(0,z,0,\ldots,0)$.
\end{itemize}
\end{proof}
%---------------------------------------------------------------------------------------%
%---------------------------------------------------------------------------------------%
%---------------------------------------------------------------------------------------%
Finally, we prove that if $R$ is an infinite ring with some restrictions, then the domination number of $TD(R,n)$ is also infinite.\\
%---------------------------------------------------------------------------------------%
The following well-known lemma is the key for the rest of this section.
\begin{lem}\label{infinitetotal}
Let $V$ be a vector space over a field $\mathbb{F}$. If $V$
is written as union of $k$ proper subspaces of $V$, then $k\geq |\mathbb{F}|$.
In particular, if $\mathbb{F}$ is an infinite field, then $V$ cannot be written as union of a finite number of proper subspaces.
\end{lem}
Here $H_{\mathbf{a}}$ denote the hyperplane $\mathbf{a}\cdot \mathbf{x}=0$.
%---------------------------------------------------------------------------------------%
\begin{thm}\label{fielddomin}
Let $\mathbb{F}$ be a field. Then $\gamma(TD(\mathbb{F},n))$ is finite if and only if $\mathbb{F}$ is a finite field.
\end{thm}
%---------------------------------------------------------------------------------------%
\begin{proof}
On the contrary, assume that $D=\{\mathbf{a}_1,\ldots,\mathbf{a}_k \}$ is a dominating set for $TD(\mathbb{F},n)$. We know that $N(\mathbf{a}_i)\subset H_{\mathbf{a}_i}$. Therefore, $\bigcup_{i=1}^{k} N(\mathbf{a}_i)\subset \bigcup_{i=1}^{k} H_{\mathbf{a}_i}$. Then by Lemma \ref{infinitetotal}, it follows that $\mathbb{F}$ should be finite.
\end{proof}
%---------------------------------------------------------------------------------------%
%---------------------------------------------------------------------------------------%
\begin{lem}
Let $R$ be a ring and $m$ be a maximal ideal such that $R/m$ has infinitely many elements. Then $R^{n}$ cannot be written as union of a finite number of proper $R$-submodules.
\end{lem}
%---------------------------------------------------------------------------------------%
\begin{proof}
On the contrary, assume that $R^{n}=\bigcup_{i=1}^{d}V_{i}$, where $V_{i}$ are $R$-submodules of $R^{n}$. It is known that $R^{n}\otimes_{R} R/m$ is a vector space over the field $R/m$. Then $$R^{n}\otimes_{R} R/m=\Big(\bigcup_{i=1}^{d}V_{i}\otimes_{R} R/m \Big)=\bigcup_{i=1}^{d}(V_{i}\otimes_{R} R/m).$$ Since $R/m$ is infinite and $V_{i}\otimes_{R} R/m$ are vectore subspaces, we get contradiction by Lemma \ref{infinitetotal}.
\end{proof}
%---------------------------------------------------------------------------------------%
\begin{thm}
Let $R$ be a ring and $m$ be a maximal ideal such that $R/m$ has infinite elements. Then $\gamma(TD(R,n))$ is not finite.
\end{thm}
%---------------------------------------------------------------------------------------%
%---------------------------------------------------------------------------------------%
\begin{proof}
 The proof is similar to that of Theorem \ref{fielddomin}.
\end{proof}
%---------------------------------------------------------------------------------------%
%---------------------------------------------------------------------------------------%
\begin{rmk}
Let $R$ be a ring such that $\sup \big\{ |R/m|\mid m~\text{\rm{is a maximal ideal of}}~R \big\}=\infty$. Then $\gamma(TD(R,n))$ is not finite. Rings $\mathbb{Z}$ and $\mathbb{F}[x]$ are such examples.
\end{rmk}
%---------------------------------------------------------------------------------------%
%---------------------------------------------------------------------------------------%
More generally, we have the following theorem:
%---------------------------------------------------------------------------------------%
\begin{thm}
Let $R$ be a ring. Let $\nu =\sup \big\{ |R/m|\mid m\text{\rm{ is a maximal ideal of}}~R \big\}$. Then $\gamma(TD(R,n))\geq \nu$.
\end{thm}
%---------------------------------------------------------------------------------------%
%---------------------------------------------------------------------------------------%
It would be desirable to show that for an arbitrary infinite ring, the domination number is infinite. 
%---------------------------------------------------------------------------------------%
%---------------------------------------------------------------------------------------%
\section{Independence and clique number}
Our aim in this section is to investigate the clique and independence number of $TD(R,n)$.\\
%---------------------------------------------------------------------------------------%
By Theorem \ref{n2}, the next theorem about the clique and independence number for $n=2$ follows immediately.
%---------------------------------------------------------------------------------------%
%---------------------------------------------------------------------------------------%
\begin{thm}\label{cliquen2}
Let $\mathbb{F}$ be a finite field. Then 
\begin{equation}\omega(TD(\mathbb{F},2))=\left\{
\begin{array}{cc}
|\mathbb{F}|-1 &~\rm{if}~|\mathbb{F}|\equiv 3~(\bmod~ 4),\\
2 & \textrm{\rm otherwise}.\
\end{array} \right.
\end{equation}
Also, for the independence number we have the following:
\begin{equation}\alpha(TD(\mathbb{F},2))=\left\{
\begin{array}{cc}
\dfrac{O(2,\mathbb{F})}{|\mathbb{F}|-1}+\dfrac{(|\mathbb{F}|^2-1)-O(2,\mathbb{F})}{2} &~\rm{if}~ O(\mathbb{F},2)\neq 0,\\
\dfrac{|\mathbb{F}|^2-1}{2} & \rm{otherwise}.\
\end{array} \right.
\end{equation}
\end{thm}
%---------------------------------------------------------------------------------------%
The set $\{e_1,\ldots,e_n\}$ is a clique in $TD(R,n)$. Then $\omega(TD(R,n))\geq  n$. In the next theorem, we prove that under some conditions we have equality.
%---------------------------------------------------------------------------------------%
%---------------------------------------------------------------------------------------%
\begin{thm}
Let $R$ be an integral domain such that $O(R,n)=0$. Then $\omega(TD(R,n))= n$.
\end{thm}
%---------------------------------------------------------------------------------------%
\begin{proof}
Let $W=\{a_1,\ldots,a_t\}$ be a clique in $TD(R,n)$. We show that $a_{1},\ldots,a_t$ should be linearly independent over $R$. Let $$\alpha_{1}a_1+\cdots+\alpha_{t}a_t=0.$$ Therefore, by multiplying to $a_i$ for $i=1,\ldots,t$, we have
$$\alpha_i ||a_i||=0.$$
Then $\alpha_i=0$. Since $R^{n}$ is a free $R$-module, $t\leq n$.
\end{proof}
%---------------------------------------------------------------------------------------%
If we drop the condition $O(R,n)=0$, above theorem is no longer hold. The next two theorems show that if $O(R,n)\neq 0$, then the clique number is exponentially large.
%---------------------------------------------------------------------------------------%
\begin{thm}
Let $\mathbb{F}$ be a field such that $O(\mathbb{F},2)\neq 0$. Then 
\begin{equation} \label{cliquebound}
\omega(TD(\mathbb{F},n))\geq |\mathbb{F}|^{[\frac{n}{2}]}-1.
\end{equation}
If $n$ is an odd number, then $\omega(TD(\mathbb{F},n))\geq |\mathbb{F}|^{[\frac{n}{2}]}$.
\end{thm}
%---------------------------------------------------------------------------------------%
%---------------------------------------------------------------------------------------%
%---------------------------------------------------------------------------------------%
%---------------------------------------------------------------------------------------%
\begin{proof}
Let $a^2+b^2=0$ and $(a,b)\neq (0,0)$. Let $\mathbf{a}_i=ae_1+be_2+\cdots+ae_{2i-1}+be_{2i}$, for $i=1,\ldots,[\frac{n}{2}]$. Let $W$ be the vector subspace generated by $\{ \mathbf{a}_i\mid i=1,\ldots,[\frac{n}{2}] \}$. Then $W-\{ \mathbf{0} \}$ is a clique in $TD(\mathbb{F},n)$ of size $|\mathbb{F}|^{[\frac{n}{2}]}-1$.

If $n$ is odd, then $W \cup \{ e_{n} \}$ is a clique set. Hence, $\omega(TD(\mathbb{F},n))\geq |\mathbb{F}|^{[\frac{n}{2}]}$.
\end{proof}
By Theorem \ref{cliquen2}, the inequality (\ref{cliquebound}) can turn into equality for $n=2$.
%---------------------------------------------------------------------------------------%
%---------------------------------------------------------------------------------------%
\begin{thm}
Let $R$ be a ring such that $O(R,2)\neq 0$. Then 
\begin{equation} \label{cliqueboundring}
\omega(TD(R,n))\geq 2^{[\frac{n}{2}]}-1.
\end{equation}
If $n$ is an odd number, then $\omega(TD(R,n))\geq 2^{[\frac{n}{2}]}$.
\end{thm}
%---------------------------------------------------------------------------------------%
%---------------------------------------------------------------------------------------%
%---------------------------------------------------------------------------------------%
%---------------------------------------------------------------------------------------%
\begin{proof}
Let $a^2+b^2=0$ and $(a,b)\neq (0,0)$. Let $\mathbf{a}_i=ae_1+be_2+\cdots+ae_{2i-1}+be_{2i}$, for $i=1,\ldots,[\frac{n}{2}]$. Let $$W=\{\sum_{i=1}^{[\frac{n}{2}]} \varepsilon_{i}\mathbf{a}_i \mid \varepsilon_{i} \in \{0,1\} \}.$$ Then $W-\{ \mathbf{0} \}$ is a clique in $TD(R,n)$ of size $2^{[\frac{n}{2}]}-1$.

If $n$ is odd, then $W \cup \{ e_{n} \}$ is a clique set. Hence, $\omega(TD(R,n))\geq 2^{[\frac{n}{2}]}$.
\end{proof}
%---------------------------------------------------------------------------------------%
%---------------------------------------------------------------------------------------%
\begin{rmk}
By Proposition 6.1 of \cite{akhtar}, one can easily find a better lower bound for the inequality (\ref{cliqueboundring}). 
\end{rmk}
%---------------------------------------------------------------------------------------%
%---------------------------------------------------------------------------------------%
\begin{rmk}
Let $W=\{\mathbf{a}_{1},\ldots,\mathbf{a}_{t} \}$ be a clique set of $TD(R,n)$. Then the set $\Delta=\{(\mathbf{a}_{i},1)\mid i=1,\ldots, t\} \cup \{e_{n+1}\}$ is an independent set for $TD(R,n+1)$. Hence, $$\omega(TD(R,n))+1\leq \alpha(TD(R,n+1)).$$
Let $W=\{\mathbf{a}_{1},\ldots,\mathbf{a}_{t} \}$ be a clique set of $TD(R,n)$ such that $||\mathbf{a}_{i}||=0$ for $i=1,\ldots, t$. Then the set $\Delta=\{(\mathbf{a}_{i},\beta)\mid i=1,\ldots, t~and~\beta \in R^{\ast}\}\cup \{\beta e_{n+1}\mid \beta \in R^{\ast}\}$ is an independent set for $TD(R,n+1)$. Hence, $$(|R|-1) (\omega(TD(R,n))+1)\leq \alpha(TD(R,n+1)).$$
\end{rmk}
%---------------------------------------------------------------------------------------%
%---------------------------------------------------------------------------------------%
The following proposition about tensor product of graphs is straightforward.
%---------------------------------------------------------------------------------------%
\begin{prop}
Let $G$ and $H$ be simple graphs. Then $\omega(G\otimes H) = \min \{ \omega(G),\omega(H)\}$.
\end{prop}
%---------------------------------------------------------------------------------------%
\begin{dfn}
A clique-loop is a set of pairwise adjacent vertices in a graph, and loop at each vertex. Let us denote by $\overline{\omega}(G)$ the size of the largest clique-loop of $G$.
\end{dfn}
%---------------------------------------------------------------------------------------%
The following proposition can be proved easily.
%---------------------------------------------------------------------------------------%
%---------------------------------------------------------------------------------------%
\begin{prop}
Let $G$ and $H$ be graphs. Then $\omega(G\otimes H)\geq \overline{\omega}(G\otimes H) =  \overline{\omega}(G)\overline{\omega}(H)$
\end{prop}
%---------------------------------------------------------------------------------------%
It is easy to check that $\omega(TD(R,n))=\omega(\overline{TD(R,n)})-1$.
%---------------------------------------------------------------------------------------%
However, it seems difficult to find the clique number of $TD(R,n)$, for an arbitrary ring $R$ and integer $n$.
%---------------------------------------------------------------------------------------%
\begin{thm}
Let $\mathbb{F}$ be a field such that $O(\mathbb{F},2)\neq 0$. Then $\overline{\omega}(\overline{TD(\mathbb{F},n)})=|\mathbb{F}|^{[\dfrac{n}{2}]}$.
\end{thm}
%---------------------------------------------------------------------------------------%
%---------------------------------------------------------------------------------------%
%---------------------------------------------------------------------------------------%
%---------------------------------------------------------------------------------------%
\begin{proof}
Let $\Delta$ be a clique-loop of maximum size. We first prove that $\Delta$ is a vector subspace. Since $\Delta$ is maximum, then $\mathbf{0}\in \Delta$, and if $\mathbf{a},\mathbf{b}\in \Delta$, then $\mathbf{a}-\mathbf{b}\in \Delta$. Since for all $\mathbf{a},\mathbf{b}\in \Delta$, we have $\mathbf{a} \cdot \mathbf{b}= 0$, it follows that $\Delta\subseteq \Delta^{\bot}$. Therefore, $\dim_{\mathbb{F}}\Delta \leq [\dfrac{n}{2}]$ completes the proof.
\end{proof}
%---------------------------------------------------------------------------------------%
%---------------------------------------------------------------------------------------%
%---------------------------------------------------------------------------------------%
Now, we will show that the ring $R$ is finite if and only if $\alpha(TD(R,n))$ is finite.
%---------------------------------------------------------------------------------------%
%---------------------------------------------------------------------------------------%
\begin{thm}
Let $R$ be an infinite ring. Then $\alpha(TD(R,n))=\infty$.
\end{thm}
%---------------------------------------------------------------------------------------%
\begin{proof}
Let $U(R)=R_{1}\cup R_{2}$ be a partition of invertible elements in such a way that, if $a\in R_{i}$, then $-a^{-1}\notin R_{i}-\{a\}$, for $i=1,2$.
Without restriction of generality, we can assume $|R_{1}|\geq |R_{2}|$.
Let $\mathfrak{R}:=(R-U(R))\cup R_1$. It means that if $x$ and $y$ are distinct elements of $\mathfrak{R}$, then $xy\neq -1$. We see at once that $|\mathfrak{R}|=\infty$. Let $\Delta=\{e_1+ae_{2}\mid a\in \mathfrak{R}\}$. Hence, $\Delta$ is an independent set with infinitely many elements.
\end{proof}
%---------------------------------------------------------------------------------------%
%---------------------------------------------------------------------------------------%
%---------------------------------------------------------------------------------------%
%%---------------------------------------------------------------------------------------%
%---------------------------------------------------------------------------------------%
%---------------------------------------------------------------------------------------%
\section{Planarity}
%---------------------------------------------------------------------------------------%
%---------------------------------------------------------------------------------------%
In \cite{Planar}, the authors have classified all finite commutative rings $R$ such that $\Gamma(R)$ is planar. In this section, we classify all rings $R$ and $n$, such that $TD(R,n)$ is planar.
%---------------------------------------------------------------------------------------%

A remarkable characterization of the planar graphs was given by Kuratowski in 1930.
%---------------------------------------------------------------------------------------%
%---------------------------------------------------------------------------------------%
%---------------------------------------------------------------------------------------%
%---------------------------------------------------------------------------------------%
\begin{thm}
A finite graph is planar if and only if it does not contain a subgraph that is a subdivision of $K_{5}$ or $K_{3,3}$.
\end{thm}
%---------------------------------------------------------------------------------------%
The next theorem classifies all planar graphs $TD(R,n)$.
%---------------------------------------------------------------------------------------%
%---------------------------------------------------------------------------------------%
\begin{thm}
Let $R$ be a commutative ring and $n$ be a natural number. Then $G=TD(R,n)$ is planar if and only if $G=TD(\mathbb{Z}_2,2)$, $G=TD(\mathbb{Z}_2,3)$ or $G=TD(\mathbb{Z}_3,2)$.
\end{thm}
%---------------------------------------------------------------------------------------%
\begin{proof}
Let $TD(R,n)$ be a planar graph. Since $\omega(TD(R,n))\geq n$, we have $n\leq 4$.\\
Let $R$ be a ring with at least 4 elements. Let $a,b,c$ be three distinct non-zero elements of $R$. Therefore, the graph $K_{3,3}$ is a subgraph induced by $\{ae_{1},be_{1},ce_{1},ae_{2},be_{2},ce_{2}\}$. Then $R$ is either $\mathbb{Z}_2$ or $\mathbb{Z}_3$, and $n\leq 4$.

Let $n=4$. It is easy to check that the graph $K_{3,3}$ is a subgraph induced by $\{e_{1},e_{2},e_{1}+e_{2},e_{3},e_{4},e_{3}+e_{4}\}$.
Let $G=TD(\mathbb{Z}_3,3)$. Let $H$ be the subgraph of $G$ induced by $\{ e_1,e_2,e_3,2e_1,2e_2,2e_3, e_2+e_3,e_1+e_3,e_1+e_2 \}$. Merge $2e_1$ and $e_2+e_3$, $2e_2$ and $e_1+e_3$, and $2e_3$ and $e_1+e_2$. The new graph is isomorphic to the $K_{3,3}$ graph. Then $G$ cannot be planar.

It is easy to check that the graphs $TD(\mathbb{Z}_2,2)$, $TD(\mathbb{Z}_2,3)$ and $TD(\mathbb{Z}_3,2)$ are planar graphs.
\end{proof}
%---------------------------------------------------------------------------------------%
%---------------------------------------------------------------------------------------%
%---------------------------------------------------------------------------------------%
%---------------------------------------------------------------------------------------%
%---------------------------------------------------------------------------------------%
%{\bf Acknowledgments.}\\ The research of.
%---------------------------------------------------------------------------------------%
%---------------------------------------------------------------------------------------%
%---------------------------------------------------------------------------------------%
%---------------------------------------------------------------------------------------%
{}

\end{document}